\documentclass[10pt]{article}
\usepackage{indentfirst, latexsym, amsmath,bm,amssymb}
\usepackage{tikz}
\usepackage{color,soul}
\usepackage{amsmath}
\usepackage{amssymb}
\usepackage{mathrsfs}
\usepackage{mathtools}
\usepackage{stmaryrd}

\begin{document}

\newcommand{\Ecc}{\operatorname{Ecc}}
\newcommand{\ecc}{\operatorname{ecc}}
\newcommand{\Norm}{\operatorname{Norm}}
\newcommand{\norm}{\operatorname{norm}}
\newcommand{\diam}{\operatorname{diam}}
\newcommand{\rad}{\operatorname{rad}}
\newcommand{\degree}{\operatorname{deg}}
\renewcommand{\deg}{\overline{d}} 

\newtheorem{theo}{Theorem}[section]
\newtheorem{theorem}{Theorem}[section]
\newtheorem{definition}[theorem]{Definition}
\newtheorem{prop}[theo]{Proposition}
\newtheorem{lemma}[theo]{Lemma}
\newtheorem{claim}[theo]{Claim}
\newtheorem{cor}[theo]{Corollary}
\newtheorem{question}{Question}[section]
\newtheorem{remark}[theo]{Remark}
\newtheorem{defi}{Definition}
\newtheorem{conj}[theo]{Conjecture}
\newtheorem{ob}[theo]{Observation}
\newtheorem{corollary}[theorem]{Corollary}
\def\qedsymbol{\ensuremath{\scriptstyle\blacksquare}}
\newenvironment{proof}{\noindent {\bf
Proof.}}{\rule{3mm}{3mm}\par\medskip}

\renewcommand{\deg}{\overline{d}}
\newcommand{\JEC}{{\it Europ. J. Combinatorics},  }
\newcommand{\JCTB}{{\it J. Combin. Theory Ser. B.}, }
\newcommand{\JCT}{{\it J. Combin. Theory}, }
\newcommand{\JGT}{{\it J. Graph Theory}, }
\newcommand{\ComHung}{{\it Combinatorica}, }
\newcommand{\DM}{{\it Discrete Math.}, }
\newcommand{\ARS}{{\it Ars Combin.}, }
\newcommand{\SIAMDM}{{\it SIAM J. Discrete Math.}, }
\newcommand{\SIAMADM}{{\it SIAM J. Algebraic Discrete Methods}, }
\newcommand{\SIAMC}{{\it SIAM J. Comput.}, }
\newcommand{\ConAMS}{{\it Contemp. Math. AMS}, }
\newcommand{\TransAMS}{{\it Trans. Amer. Math. Soc.}, }
\newcommand{\AnDM}{{\it Ann. Discrete Math.}, }
\newcommand{\NBS}{{\it J. Res. Nat. Bur. Standards} {\rm B}, }
\newcommand{\ConNum}{{\it Congr. Numer.}, }
\newcommand{\CJM}{{\it Canad. J. Math.}, }
\newcommand{\JLMS}{{\it J. London Math. Soc.}, }
\newcommand{\PLMS}{{\it Proc. London Math. Soc.}, }
\newcommand{\PAMS}{{\it Proc. Amer. Math. Soc.}, }
\newcommand{\JCMCC}{{\it J. Combin. Math. Combin. Comput.}, }
\newcommand{\GC}{{\it Graphs Combin.}, }
\title{The normality and sum of normalities of trees\thanks{
This work is supported by the National Natural Science Foundation of China (Nos. 11601208, 11531001, 11601337 and 11271256) and  the Joint NSFC-ISF Research Program (jointly funded by the National Natural Science Foundation of China and the Israel Science Foundation (No. 11561141001)).
  }}
\author{ Ya-Hong  Chen$^{1}$,  Hua Wang$^{2,3}$, Xiao-Dong Zhang$^4$\thanks{Corresponding  author ({\it E-mail address: xiaodong@sjtu.edu.cn}).}
\\
{\small $^1$Department of Mathematics},
{\small Lishui University} \\
{\small  Lishui, Zhejiang 323000, PR China}\\
{\small $^2$ College of Software, Nankai University}\\
{\small Tianjin 300071, P.R. China} \\
{\small $^3$Department of Mathematical Sciences,}
{\small Georgia Southern University }\\
{\small Statesboro, GA 30460, USA}\\
{\small $^4$ School of Mathematical Sciences, MOE-LSC, SHL-MAC,}
{\small Shanghai Jiao Tong University} \\
{\small  800 Dongchuan road, Shanghai, 200240,  P.R. China}
}

\date{}
\maketitle
 \begin{abstract}
The eccentricity of a vertex $v$ in a graph $G$ is the maximum distance from $v$ to any other vertex. The vertices whose eccentricity are equal to the diameter (the maximum eccentricity) of $G$ are called peripheral vertices. In trees the eccentricity at $v$ can always be achieved by the distance from $v$ to a peripheral vertex. From this observation we are motivated to introduce normality of a vertex $v$ as the minimum distance from $v$ to any peripheral vertex. We consider the properties of the normality as well as the middle part of a tree with respect to this concept. Various related observations are discussed and compared with those related to the eccentricity. Then, analogous to the sum of eccentricities we consider the sum of normalities. After briefly discussing the extremal problems in general graphs we focus on trees and trees under various constraints. As opposed to the path and star in numerous extremal problems, we present several interesting and unexpected extremal structures. Lastly we consider the difference between eccentricity and normality, and the sum of these differences. We also introduce some unsolved problems in the context.
 \end{abstract}


\vskip 0.5cm

\section{Introduction}

The {\it eccentricity} of a vertex $v$ in a connected graph $G$ is the largest distance from $v$ to any other vertex, i.e.
\begin{equation}\label{eq:ecc}
\ecc_G(v)=\max_{u\in V(G)}d(v,u)
\end{equation}
where $d_G(v,u)$ (or simply $d(v,u)$) is the distance between $v$ and $u$ in $G$.
Very often, when it is clear what the underlying graph is, we simply write $\ecc(v)$ without identifying $G$.

The {\it diameter} of a graph $G$ is the largest eccentricity:
$$ \diam(G)=\max_{v\in V(G)}\ecc_G(v). $$

Similarly, the {\it radius} of a graph $G$ is the smallest eccentricity:
$$\rad(G) = \min_{v\in V(G)}\ecc_G(v). $$

The eccentricity, along with diameter and radius, are among the best-known concepts defined on distances in graphs \cite{chen,2014Danke,1981hedet,1989rose,2016smith,2017xu}. It is known that the eccentricity of a vertex, considered as a function of the vertex, is concave up along any path in a tree \cite{1869jordan}. Consequently, the eccentricity of a vertex in a tree is maximized at a leaf and minimized at one or two adjacent vertices.
More generally, a vertex $v$ of $G$ with $\ecc_G(v)=\diam(G)$ is called a {\it peripheral vertex} of $G$ and the set of peripheral vertices of $G$ is called its {\it periphery} and is denoted by $P(G)$. The sum of distances between peripheral vertices has been proposed as a variation of a well-known chemical index called the Wiener index (defined as the sum of the distances between all vertices) \cite{chen}. On the other hand, the {\it center} of a graph is the collection of vertices with $\ecc_G(v)=\rad(G)$, denoted by $C(G)$. The center of a graph and its properties has been studied as early as in \cite{1869jordan}.

The sum of eccentricities of a graph is defined as
$$ \Ecc(G)=\sum_{v\in V(G)} \ecc_G(v) .$$
Compared with $\ecc(v)$ (considered as a local function on vertices), $\Ecc(G)$ is a global function or graph invariant. As it is the case with many graph invariants, it is interesting to consider extremal problems that seek graph structures that maximize or minimize a given invariant. For trees this has been extensively examined \cite{2016smith}.

In this paper we also mostly discuss trees. First we note that the eccentricity at a vertex in a tree $T$ can also be viewed as
\begin{equation}\label{eq:ecc'}
\ecc_T(v) = \max_{u\in P(T)}d(v,u) .
\end{equation}
That is, in a tree the largest distance from any vertex is always obtained with some peripheral vertex. Note that this is not true for general graphs, as shown in Figure~\ref{fig:counter1}.

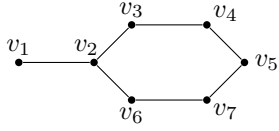
\begin{figure}[htbp]
\centering
    \begin{tikzpicture}[scale=1]
        \node[fill=black,circle,inner sep=1pt] (t1) at (0,0) {};
        \node[fill=black,circle,inner sep=1pt] (t2) at (1,0) {};
        \node[fill=black,circle,inner sep=1pt] (t3) at (1.5,0.5) {};
       \node[fill=black,circle,inner sep=1pt] (t4) at (1.5,-0.5) {};
\node[fill=black,circle,inner sep=1pt] (t5) at (2.5,0.5) {};
       \node[fill=black,circle,inner sep=1pt] (t6) at (2.5,-0.5) {};
 \node[fill=black,circle,inner sep=1pt] (t7) at (3,0) {};

        \draw (t1)--(t2);

        \draw (t2)--(t3);
        \draw  (t3)--(t5);
        \draw  (t2)--(t4);
         \draw  (t6)--(t4);
         \draw  (t5)--(t7);
         \draw  (t6)--(t7);
   \node at (0,.2) {$v_{1}$};
         \node at (.9,.2) {$v_{2}$};
         \node at (1.5,.7) {$v_{3}$};
         \node at (2.75,.65) {$v_{4}$};
        \node at (3.3,0) {$v_{5}$};
\node at (1.5,-.7) {$v_{6}$};
         \node at (2.75,-.65) {$v_{7}$};

        \end{tikzpicture}
\caption{A graph $G$ with $P(G)=\{v_1, v_5\}$ and $\ecc_G(v_3) = d(v_3,v_7) > d(v_3,v_1)=d(v_3,v_5)$. }\label{fig:counter1}
\end{figure}

The equivalence between \eqref{eq:ecc} and \eqref{eq:ecc'} in a tree will be discussed in Section~\ref{sec:defi}, where we also introduce the natural variation called the {\it normality} at a vertex:
\begin{equation}\label{eq:norm}
\norm_T(v) = \min_{u\in P(T)}d(v,u)
\end{equation}
The normality of a vertex measures how close a vertex is to the periphery of the graph. Consequently the normality at a peripheral vertex is zero. Simple observations related to the normality are also presented in Section~\ref{sec:defi}.

Similar to $\Ecc(G)$, from the normalities at individual vertices we may also define the sum of normalities:
$$ \Norm(G) = \sum_{v\in V(G)} \norm_G(v) . $$
Extremal problems with respect to $\Norm(\cdot)$ in trees will be investigated in Section~\ref{sec:ex}.

In general graphs, some common properties of distance-based graph invariants fail for $\Norm(\cdot)$. Take, for instance, the tree $T$ as shown in Figure~\ref{fig:counter2}, we have
\begin{itemize}
\item $P(T)=\{v_1,v_4,v_5\}$ and $\Norm(T)=1+1 = 2$;
\item $P(T+v_1v_5)=\{v_1,v_4\}$ and $\Norm(T+v_1v_5)=1+1+1=3$.
\end{itemize}
Thus adding an edge, which decreases the distances between vertices, increased the sum or normalities.

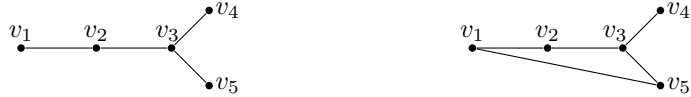
\begin{figure}[htbp]
\centering
    \begin{tikzpicture}[scale=1]
        \node[fill=black,circle,inner sep=1pt] (t1) at (0,0) {};
        \node[fill=black,circle,inner sep=1pt] (t2) at (1,0) {};
        \node[fill=black,circle,inner sep=1pt] (t3) at (2,0) {};
        \node[fill=black,circle,inner sep=1pt] (t4) at (2.5,0.5) {};

        \node[fill=black,circle,inner sep=1pt] (t5) at (2.5,-0.5) {};

  \node[fill=black,circle,inner sep=1pt] (t6) at (6,0) {};
        \node[fill=black,circle,inner sep=1pt] (t7) at (7,0) {};
        \node[fill=black,circle,inner sep=1pt] (t8) at (8,0) {};
        \node[fill=black,circle,inner sep=1pt] (t9) at (8.5,0.5) {};

        \node[fill=black,circle,inner sep=1pt] (t10) at (8.5,-0.5) {};

        \draw (t1)--(t2);

        \draw (t2)--(t3);
        \draw  (t3)--(t4);
        \draw  (t3)--(t5);

\draw (t6)--(t7);

        \draw (t7)--(t8);
        \draw  (t8)--(t9);
        \draw  (t8)--(t10);
  \draw  (t6)--(t10);
   \node at (0,.2) {$v_{1}$};
         \node at (1,.2) {$v_{2}$};
         \node at (1.95,.2) {$v_{3}$};
         \node at (2.75,.5) {$v_{4}$};
        \node at (2.75,-.5) {$v_{5}$};

         \node at (6,.2) {$v_{1}$};
         \node at (7,.2) {$v_{2}$};
         \node at (7.9,.2) {$v_{3}$};
         \node at (8.75,.5) {$v_{4}$};
        \node at (8.75,-.5) {$v_{5}$};

        \end{tikzpicture}
\caption{The tree $T$ (on the left) and the graph $T+v_1v_5$ (on the right).}\label{fig:counter2}
\end{figure}

In this particular case it is because of the change in the number of peripheral vertices. Consequently it makes sense to study the extremal problems in trees with a given number of peripheral vertices. This will also be done in Section~\ref{sec:ex}.

Another natural quantity is the difference between the eccentricity and normality at a vertex, denoted by
$$ \lambda_T(v) = \ecc_T(v) - \norm_T(v). $$
This measures the ``span'' of distances from a vertex to the periphery of the graph. Then the ``span'' of a graph is simply the sum of these quantities over all vertices, denoted by
$$ \Lambda(T) = \sum_{v \in V(T)} \lambda_T(v) . $$ A preliminary study of these concepts are provided in Section~\ref{sec:lambda}.

As is the case with the eccentricity, we often drop the subscript $G$ or $T$ if it is clear what the underlying graph or tree is.

\section{The eccentricity and normality of a vertex}
\label{sec:defi}

First we show that the definitions in \eqref{eq:ecc} and \eqref{eq:ecc'} are indeed equivalent in trees.

To see this, first we note the following fact.

\begin{prop}
\label{prop:fact}
In a tree $T$ with a longest path of length $d=\diam(T)$ from $u$ to $w$, the eccentricity of any vertex $v$ is either $d(v,u)$ or $d(v,w)$.
\end{prop}

Now it suffices to show the following.

\begin{prop}
\label{prop:12}
In a tree $T$, the eccentricity of any vertex $v$ is obtained between $v$ and some peripheral vertex.
\end{prop}

\begin{proof}
Let the path $P(u,w)$ from $u$ to $w$ be one of the longest paths in $T$, then $d(u,w)=\diam(T)$ and $u,w \in P(T)$. Then by Proposition~\ref{prop:fact} the eccentricity of any vertex $v$ is $d(v,u)$ or $d(v,w)$, i.e. obtained between $v$ and at least one of $u$ and $w$ (hence some peripheral vertex).
\end{proof}

With the equivalent definition \eqref{eq:ecc'} of the eccentricity, we now examine the analogous concept \eqref{eq:norm}. It is easy to see that the normality of a vertex is minimized (equals zero) at peripheral vertices. A natural question is to find the collection of vertices that maximize the normality in a tree $T$:
$$ C_{\norm}(T) = \{ v\in V(G) | \norm(v) = \max_{u \in V(G)} \norm(u) \} . $$

There are many examples of trees where $C(T) = C_{\norm}(T)$, such as the path $P_n$ and the star $S_n$. On the other hand, in the tree shown in Figure~\ref{fig:cat} we have
$$ P(T) = \{v_1, v_2, v_3, v_4\} \hbox{ and } C_{\norm}(T) = \{ v_5, v_6, v_7, v_8 \} $$
while
$$ C(T) = \{ v_9, v_{10} \} . $$

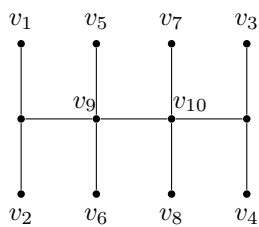
\begin{figure}[htbp]
\centering
    \begin{tikzpicture}[scale=1]
        \node[fill=black,circle,inner sep=1pt] (t1) at (0,0) {};
        \node[fill=black,circle,inner sep=1pt] (t2) at (-1,0) {};
        \node[fill=black,circle,inner sep=1pt] (t3) at (-2,0) {};
        \node[fill=black,circle,inner sep=1pt] (t4) at (1,0) {};
        \node[fill=black,circle,inner sep=1pt] (t7) at (-2,1) {};
        \node[fill=black,circle,inner sep=1pt] (t8) at (-2,-1) {};
\node[fill=black,circle,inner sep=1pt] (t9) at (-1,1) {};
\node[fill=black,circle,inner sep=1pt] (t10) at (-1,-1) {};
\node[fill=black,circle,inner sep=1pt] (t11) at (0,1) {};
\node[fill=black,circle,inner sep=1pt] (t12) at (0,-1) {};
\node[fill=black,circle,inner sep=1pt] (t13) at (1,1) {};
\node[fill=black,circle,inner sep=1pt] (t14) at (1,-1) {};

        \draw (t1)--(t2);
        \draw (t2)--(t3);

        \draw (t1)--(t4);
        \draw (t3)--(t7);

        \draw (t3)--(t8);

         \draw (t2)--(t9);

        \draw (t2)--(t10);
        \draw (t1)--(t11);

        \draw (t1)--(t12);
        \draw (t4)--(t13);
        \draw (t4)--(t14);

\node at (-2,1.3) {$v_1$};
\node at (-2,-1.3) {$v_2$};
\node at (1,1.3) {$v_3$};
\node at (1,-1.3) {$v_4$};
\node at (-1,1.3) {$v_5$};
\node at (-1,-1.3) {$v_6$};
\node at (0,1.3) {$v_7$};
\node at (0,-1.3) {$v_8$};
\node at (-1.15,.2) {$v_9$};
\node at (0.25,.2) {$v_{10}$};

\end{tikzpicture}
\caption{A tree $T$ with disconnected $C_{\norm}(T) \neq C(T)$.}\label{fig:cat}
\end{figure}

Hence $C_{\norm}(T)$ may even be disjoint from $C(T)$.
This example also shows that $C_{\norm}(T)$ is not necessarily connected.

\section{The extremal problems with respect to the sum of normalities}
\label{sec:ex}

First we introduce a special tree on $n$ vertices. Recall that an $r$-comet on $n$ vertices is obtained from attaching $n-r$ pendant edges to one end of a path on $r$ vertices. The other end of the path is called the {\it head} of the comet. Note that when $r=n-1$ this is simply a path.

\begin{defi}
\label{defi:tnd}
Let $\widehat{T}_{n,d}$ be a tree of order $n$ obtained from a path of length $d$ by identifying the head of an $r$-comet of $n-d$ vertices with the middle point (or one of the two middle points if $d$ is odd) of the path, where
$$ r= \min \left\{ \left\lfloor \frac{d}{2} \right\rfloor -1, n-d-1 \right\}  \hbox{ (See Figure \ref{fig hat}).} $$
\end{defi}

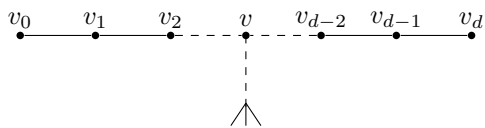
\begin{figure}[htbp]
\centering
    \begin{tikzpicture}[scale=1]
        \node[fill=black,circle,inner sep=1pt] (t1) at (0,0) {};
        \node[fill=black,circle,inner sep=1pt] (t2) at (1,0) {};
        \node[fill=black,circle,inner sep=1pt] (t3) at (2,0) {};
        \node[fill=black,circle,inner sep=1pt] (t4) at (3,0) {};

        \node[fill=black,circle,inner sep=1pt] (t5) at (4,0) {};
        \node[fill=black,circle,inner sep=1pt] (t6) at (5,0) {};
\node[fill=black,circle,inner sep=1pt] (t7) at (6,0) {};

        \draw (t1)--(t2);
        \draw (t5)--(t6);
        \draw (t7)--(t6);
        \draw [dashed](t3)--(t4);
        \draw (t2)--(t3);
        \draw [dashed] (t5)--(t4);

\draw [dashed](3,0)--(3,-1);

\draw (3,-.9)--(2.8,-1.2);
 \draw (3,-.9)--(3,-1.2);
\draw (3,-.9)--(3.2,-1.2);

         \node at (0,.2) {$v_{0}$};
         \node at (2,.2) {$v_{2}$};
        \node at (3,.2) {$v$};

        \node at (6,.2) {$v_{d}$};
        \node at (1,.2) {$v_{1}$};
        \node at (5,.2) {$v_{d-1}$};
         \node at (4,.2) {$v_{d-2}$};

        \end{tikzpicture}
\caption{The tree $\widehat{T}_{n,d}$ with order $n$ and diameter $d$}\label{fig hat}
\end{figure}

Note that when $d$ is large (i.e. when $n-d \leq \left\lfloor \frac{d}{2} \right\rfloor$) the comet attached in the middle is simply a path.

For a general connected graph $G$ of order $n$, it is easy to see that
$$ \Norm(G) \geq 0 $$
with equality if and only if every vertex is in $P(G)$. This is obviously the case with the complete graph $K_n$ or the cycle $C_n$.

On the other hand, because of the example in Figure~\ref{fig:counter2} it is not obvious that the maximum $\Norm(G)$ must be obtained by a tree. However, among connected graphs of order $n$ and diameter $d$ it seems that the $\widehat{T}_{n,d}$ indeed maximizes the sum of normalities. We will verify this for trees.

\subsection{Extremal $\Norm(T)$ in trees}

First we note the following simple observation.

\begin{prop}
For any tree $T$ with order $n\geq 3$,
$$\Norm(T)\geq 1$$
with equality if and only if $T\cong S_n$.
\end{prop}
\begin{proof}
In any tree of order $n\geq 3$, there is at least one internal vertex $v$ that is not a peripheral vertex. Then
$$ \Norm(G) \geq \norm(v)\geq 1$$
with equality if and only if there is exactly one internal vertex that is adjacent to peripheral vertices. This happens if and only if $T$ is the star.
\end{proof}

Now to maximize $\Norm(T)$, let us first consider trees of a given order with a prefixed diameter.

\begin{theorem}\label{theo:tnd}
Among all trees of order $n$ with given diameter $d$,
$\Norm(T)$ is maximized by $\widehat{T}_{n,d}$.
\end{theorem}

To prove Theorem~\ref{theo:tnd}, we first introduce some observations on the characteristics of such an extremal tree. Sometimes we only deal with even $d$ as the case for odd $d$ is similar.

Let $P:=v_0v_1 \ldots v_d$ be a longest path with $v_0, v_d \in P(T)$ and let $T_1,T_2,\cdots,T_{d-1}$ be the components containing $v_1,v_2,\cdots,v_{d-1}$, respectively, in $T-E(P)$ (Figure~\ref{fig T}). The next observation claims that each $T_i$ must be a path (if $|V(T_i)|$ is small) or a comet with its head identified with $v_i$. Note that in the case of $T_1$ and $T_{d-1}$, the ``degenerated'' comet is simply a star.

\begin{figure}[htbp]
\centering
    \begin{tikzpicture}[scale=1]
        \node[fill=black,circle,inner sep=1pt] (t1) at (0,0) {};
        \node[fill=black,circle,inner sep=1pt] (t2) at (1,0) {};
        \node[fill=black,circle,inner sep=1pt] (t3) at (2,0) {};
        \node[fill=black,circle,inner sep=1pt] (t8) at (3,0) {};

        \node[fill=black,circle,inner sep=1pt] (t5) at (4,0) {};
        \node[fill=black,circle,inner sep=1pt] (t6) at (5,0) {};

        \draw (t1)--(t2);
        \draw (t5)--(t6);
        \draw (t3)--(t8);
        \draw [dashed](t2)--(t3);
        \draw [dashed] (t5)--(t8);

        \draw (t8)--(2.7, -.8)--(3.3, -.8)--cycle;
        \draw (t3)--(1.7, -.8)--(2.3, -.8)--cycle;
        \draw (t2)--(.7, -.8)--(1.3, -.8)--cycle;
        \draw (t5)--(3.7, -.8)--(4.3, -.8)--cycle;

         \node at (0,.2) {$v_{0}$};
         \node at (2,.2) {$v_{i-1}$};
        \node at (3,.2) {$v_{i}$};

        \node at (5,.2) {$v_{d}$};
        \node at (1,.2) {$v_{1}$};
        \node at (4,.2) {$v_{d-1}$};

         \node at (2,-1.2) {$T_{i-1}$};
        \node at (3,-1.2) {$T_{i}$};

        \node at (1,-1.2) {$T_{1}$};
        \node at (4,-1.2) {$T_{d-1}$};


        \end{tikzpicture}
\caption{The tree $T$ with diameter $d$ and components in $T-E(P)$.}\label{fig T}
\end{figure}
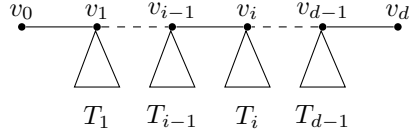

\begin{lemma}\label{lem:allpaw}
With the above notations, each $T_i$ (for $2\leq i\leq d-2$) must be:
\begin{itemize}
\item a path if $|V(T_i)| \leq \min \{i, d-i \}$;
\item an $r$-comet with $r = \min \{i, d-i \}-1$, if $|V(T_i)| > \min \{i, d-i \}$.
\end{itemize}
\end{lemma}

\begin{proof}
Consider $T_i$ and assume, without loss of generality, that $i\leq \left\lfloor \frac{d}{2} \right\rfloor $. Then $\min \{i, d-i \}=i$.

\begin{itemize}
\item If $|V(T_i)| \leq i$, it is easy to see that none of the vertices in $T_i$ are peripheral vertices, and
$$ \sum_{v \in V(T_i)} \norm_T(v) \leq i+(i+1) + \ldots + (i+|V(T_i)|-1) $$
with equality if and only if $T_i$ is a path.
\item If $|V(T_i)| > i$, first note that the distance from a vertex $v\in V(T_i)$ to $v_i$ is at most $i$ (otherwise we would have $d(v, v_d) > i + (d-i) =d$, a contradiction to the assumption of diameter).

Also note that $\norm (v) = 0$ if $d(v,v_i)=i$ (which makes $d(v,v_d)=d$ and consequently $v$ a peripheral vertex). Thus we only need to consider the vertices in $T_i$ that are at distance at most $i-1$ from $v_i$ in $\sum_{v \in V(T_i)} \norm_T(v)$. Then it is easy to see that
$$ \sum_{v \in V(T_i)} \norm_T(v) \leq i+(i+1) + \ldots + (i+(i-2)) + \underbrace{(2i-1) + (2i-1) + \ldots + (2i-1)}_{(|V(T_i)|-(i-1)) \hbox{ copies}}$$
with equality if and only if $T_i$ is an $(i-1)$-comet.
\end{itemize}
\end{proof}

Now our extremal tree is shown in Figure~\ref{fig pawsp}. We will show, next, that all but possibly one middle component ($T_{\left\lfloor \frac{d}{2} \right\rfloor}$ or $T_{\left\lceil \frac{d}{2} \right\rceil}$) are single vertex trees.

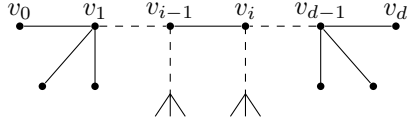
\begin{figure}[htbp]
\centering
    \begin{tikzpicture}[scale=1]
        \node[fill=black,circle,inner sep=1pt] (t1) at (0,0) {};
        \node[fill=black,circle,inner sep=1pt] (t2) at (1,0) {};
        \node[fill=black,circle,inner sep=1pt] (t3) at (2,0) {};
        \node[fill=black,circle,inner sep=1pt] (t4) at (3,0) {};

        \node[fill=black,circle,inner sep=1pt] (t5) at (4,0) {};
        \node[fill=black,circle,inner sep=1pt] (t6) at (5,0) {};

       \node[fill=black,circle,inner sep=1pt] (t7) at (1,-.8) {};
       \node[fill=black,circle,inner sep=1pt] (t8) at (4,-.8) {};

       \node[fill=black,circle,inner sep=1pt] (s7) at (.3,-.8) {};
       \node[fill=black,circle,inner sep=1pt] (s8) at (4.7,-.8) {};

        \draw (t1)--(t2);
        \draw (t5)--(t6);
        \draw (t3)--(t4);
        \draw [dashed](t2)--(t3);
        \draw [dashed] (t5)--(t4);

\draw (t2)--(s7);
\draw (t5)--(s8);

\draw (1,0)--(1,-.8);
\draw (4,0)--(4,-.8);
\draw [dashed](2,0)--(2,-1);
\draw [dashed](3,0)--(3,-1);
 \draw (2,-.9)--(1.8,-1.2);
 \draw (2,-.9)--(2,-1.2);
\draw (2,-.9)--(2.2,-1.2);

\draw (3,-.9)--(2.8,-1.2);
 \draw (3,-.9)--(3,-1.2);
\draw (3,-.9)--(3.2,-1.2);

         \node at (0,.2) {$v_{0}$};
         \node at (2,.2) {$v_{i-1}$};
        \node at (3,.2) {$v_{i}$};

        \node at (5,.2) {$v_{d}$};
        \node at (1,.2) {$v_{1}$};
        \node at (4,.2) {$v_{d-1}$};

        \end{tikzpicture}
\caption{The extremal tree from Lemma~\ref{lem:allpaw}.}\label{fig pawsp}
\end{figure}

\begin{lemma}\label{lem:onepaw}
In an extremal tree shown in Figure~\ref{fig pawsp}, we must have
$$ |V(T_i)| = 1 $$
for all $i$ except for one of $T_{\left\lfloor \frac{d}{2} \right\rfloor}$ and $T_{\left\lceil \frac{d}{2} \right\rceil}$.
\end{lemma}

\begin{proof}
First note that the extremal tree under consideration has all peripheral vertices as $v_0$ or $v_d$ or in $T_1$ and $T_{d-1}$. It is easy to see that the contribution from $T_1$ and $T_{d-1}$ to $\Norm(T)$ is zero. For any $T_i$ with $i \neq \left\lfloor \frac{d}{2} \right\rfloor$ or $\left\lceil \frac{d}{2} \right\rceil$ and $2\leq i \leq d-2$, we have
\begin{equation*}
\sum_{v \in V(T_i)} \norm_T(v)  = \left\{\begin{array}{ll}
i+(i+1) + \ldots + (i+|V(T_i)|-1), & or \\
 i+(i+1) + \ldots + (i+(i-2)) + (2i-1) + (2i-1) + \ldots + (2i-1). \end{array}\right.
\end{equation*}

Assume, for contradiction (and without loss of generality), that $|V(T_i)|>1$ for some $1\leq i <  \left\lfloor \frac{d}{2} \right\rfloor$. We now consider the tree $T'$ obtained from $T$ by ``moving'' $T_i$ from $v_i$ to $v_{i+1}$ (i.e. removing all edges adjacent to $v_i$ in $T_i$ and connecting those neighbors of $v_i$ in $T_i$ with $v_{i+1}$).

Now, in $T'$, the contribution to $\Norm(\cdot)$ from the vertices in $V(T_i)-\{v_i\}$ has strictly increased from $T$ to $T'$. This is a contradiction.

Thus the only possible nontrivial $T_i$'s  are when $i = \left\lfloor \frac{d}{2} \right\rfloor$ or $\left\lceil \frac{d}{2} \right\rceil$. If $d$ is even then $\left\lfloor \frac{d}{2} \right\rfloor = \left\lceil \frac{d}{2} \right\rceil$ and we are done. If $d$ is odd, it is easy to see that moving both components to a single vertex does not change $\Norm(T)$.
\end{proof}

We can now easily justify Theorem~\ref{theo:tnd} with the above observations.

\begin{proof}
With the above notations and from Lemma~\ref{lem:onepaw}, there is only one nontrivial $T_i$ with  $i = \left\lfloor \frac{d}{2} \right\rfloor$ or $\left\lceil \frac{d}{2} \right\rceil$. Applying Lemma~\ref{lem:allpaw} to such a tree shows that the extremal structure must be $\widehat{T}_{n,d}$.
\end{proof}

With Theorem~\ref{theo:tnd}, we may now consider the value of $\Norm(\widehat{T}_{n,d})$ for different $d$ with given order $n$.

First, if $d$ is even, the above argument can be easily adapted to compute
\begin{eqnarray}\label{eq:tnd1}
\Norm(\widehat{T}_{n,d}) &=& 2(1+\cdots+\frac{d}{2}-1)+\frac{d}{2}+\frac{d}{2}+1+\cdots+\frac{d}{2}+\frac{d}{2}-2+(n-\frac{3d}{2}+1)(d-1) \nonumber \\
&=&\frac{d}{2}(\frac{d}{2}-1)+\frac{d}{2}(\frac{d}{2}-1)+(1+2+\cdots+\frac{d}{2}-2)+(n-\frac{3d}{2}+1)(d-1) \nonumber \\
&= & \frac{d(d-2)}{2}+(n-\frac{3d}{2}+1)(d-1)+\frac{(d-2)(d-4)}{8} \nonumber \\
&=&-\frac{7}{8}d^2+(n+\frac{3}{4})d-n
 \end{eqnarray}
when $n-d \geq \left\lfloor \frac{d}{2} \right\rfloor$ (i.e. $d$ is not too large). Similarly, we also have
\begin{eqnarray}\label{eq:tnd2}
\Norm(\widehat{T}_{n,d}) &=& 2(1+\cdots+\frac{d}{2}-1)+\frac{d}{2}+\frac{d}{2}+1+\cdots+\frac{d}{2}+(n-d-1) \nonumber \\
&=&\frac{d}{2}(\frac{d}{2}-1)+\frac{d}{2}(n-d)+\frac{(n-d-1)(n-d)}{2} \nonumber \\
&=&\frac{d(d-2)}{4}+\frac{(n-d)(n-1)}{2} \nonumber \\
&=&\frac{1}{4}d^2-\frac{n}{2}d+\frac{n^2-n}{2}
 \end{eqnarray}
when $n-d < \left\lfloor \frac{d}{2} \right\rfloor$.

Comparing \eqref{eq:tnd1} and \eqref{eq:tnd2} shows that $\Norm(\cdot)$ is always optimized in the first case. The same is true for the case of odd $d$ and we have, for odd $d$ that is not too large,

\begin{eqnarray}\label{eq:tnd3}
\Norm(\widehat{T}_{n,d}) &=& 2(1+\cdots+\frac{d-1}{2})+\frac{d-1}{2}+1+\cdots+\frac{d-1}{2}+\frac{d-1}{2}-1+(n-\frac{3d}{2}+\frac{1}{2})(d-1) \nonumber \\
&=&(\frac{d}{2}+\frac{1}{2})(\frac{d}{2}-\frac{1}{2})+(\frac{3d}{2}-\frac{3}{2})(\frac{d}{4}-\frac{3}{4})+(n-\frac{3d}{2}+\frac{1}{2})(d-1) \nonumber \\
&= & (d-1)(n-\frac{5d}{4}+\frac{3}{4})+\frac{(3d-3)(d-3)}{8} \nonumber \\
&=&-\frac{7}{8}d^2+(n+\frac{1}{2})d-n+\frac{3}{8}
 \end{eqnarray}

From \eqref{eq:tnd1} and \eqref{eq:tnd3} it remains to find the optimal $d$ that maximizes these expressions. We omit the specific computations.

\begin{theorem}
For any tree $T$ on $n$ vertices, we have
$$ \Norm(T) \leq \left\lfloor \frac{2n^2-4n +1}{7} \right\rfloor $$
with equality if and only if $T \cong \widehat{T}_{n,d}$, for
\begin{equation*}
d  = \left\{\begin{array}{ll}
\frac{4n}{7}, &  n\equiv 0 \mod 7; \\
\frac{4n-4}{7}~or~\frac{4n+10}{7}, & n\equiv 1 \mod 7; \\
\frac{4n+6}{7}, &  n\equiv 2 \mod 7; \\
\frac{4n+2}{7}, &  n\equiv 3 \mod 7; \\
\frac{4n-2}{7}, & n\equiv 4 \mod 7; \\
\frac{4n+8}{7}, & n\equiv 5 \mod 7; \\
\frac{4n+4}{7}, &  n\equiv 6 \mod 7. \end{array}\right.
\end{equation*}
\end{theorem}

\subsection{Extremal $\Norm(T)$ in trees with given number of peripheral vertices}

A natural extension of the previous section is to study the extremal problem among trees with a given number of peripheral vertices.
We first consider the question of maximizing $\Norm(T)$, for this purpose we define another special tree similar to $\widehat{T}_{n,d}$ in the previous section. Recall that a {\it dumbbell} $D(n,a,b)$ is a tree obtained from attaching $a$ pendant edges to one end of a path on $n-a-b$ vertices and $b$ pendant edges to the other end.

\begin{defi}
\label{defi:tnkd}
Let $\widetilde{T}_{n,k,d}$ be a tree of order $n$ with $k\geq 2$ peripheral vertices, obtained from a dumbbell $D(k+d-1, a, b)$ with $a,b\geq 1$ and $a+b=k$, by identifying the head of an $r$-comet of $n-k-d+2$ vertices with the middle point (or one of the two middle points if $d$ is odd) of the path, where
$$ r= \min \left\{ \left\lfloor \frac{d}{2} \right\rfloor -1, n-k-d+1 \right\}  \hbox{ (See Figure \ref{fig tilde}).} $$
\end{defi}

\begin{figure}[htbp]
\centering
    \begin{tikzpicture}[scale=1]
        \node[fill=black,circle,inner sep=1pt] (t1) at (0,.5) {};
        \node[fill=black,circle,inner sep=1pt] (t2) at (1,0) {};
        \node[fill=black,circle,inner sep=1pt] (t3) at (2,0) {};
        \node[fill=black,circle,inner sep=1pt] (t4) at (3,0) {};

        \node[fill=black,circle,inner sep=1pt] (t5) at (4,0) {};
        \node[fill=black,circle,inner sep=1pt] (t6) at (5,0) {};
\node[fill=black,circle,inner sep=1pt] (t7) at (6,.5) {};
 \node[fill=black,circle,inner sep=1pt] (t8) at (0,-0.5) {};
  \node[fill=black,circle,inner sep=1pt] (t9) at (6,-.5) {};
        \draw (t1)--(t2);
        \draw (t2)--(t8);
        \draw (t9)--(t6);
        \draw (t5)--(t6);
        \draw (t7)--(t6);
        \draw [dashed](t3)--(t4);
        \draw (t2)--(t3);
        \draw [dashed] (t5)--(t4);

\draw [dashed](3,0)--(3,-1);

\draw (3,-.9)--(2.8,-1.2);
 \draw (3,-.9)--(3,-1.2);
\draw (3,-.9)--(3.2,-1.2);

         \node at (2,.2) {$v_{2}$};
        \node at (3,.2) {$v$};

        \node at (1,.2) {$v_{1}$};
        \node at (5,.2) {$v_{d-1}$};
         \node at (4,.2) {$v_{d-2}$};
 \node at (.2,0) {$\vdots$};
\node at (5.8,0) {$\vdots$};

        \end{tikzpicture}
\caption{The tree $\widetilde{T}_{n,k,d}$.}\label{fig tilde}
\end{figure}
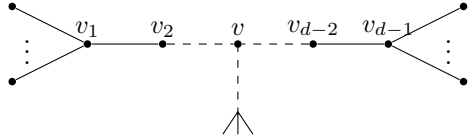

When $k=2$ it is easy to see that $\widetilde{T}_{n,k,d} = \widehat{T}_{n,d}$ in Definition~\ref{defi:tnd}. Similar to before, when $n-k-d+1 \leq \left\lfloor \frac{d}{2} \right\rfloor -1$ the ``comet'' attached in the middle is simply a path.

Following exactly the same argument we have the following analogue of Theorem~\ref{theo:tnd}.

\begin{theorem}\label{theo:tnd'}
Among all trees of order $n$ with given diameter $d$ and $k$ peripheral vertices,
$\Norm(T)$ is maximized by $\widetilde{T}_{n,k,d}$.
\end{theorem}

Again following similar computations we may find the optimal $d$. We have, however, chosen to not include the tedious computations.

\begin{theorem}
Among all trees of order $n$ with given diameter $d$ and $k$ peripheral vertices,
$\Norm(T)$ is maximized by $\widetilde{T}_{n,k,d}$ for some
$$ \left\lfloor \frac{4(n-k)+10}{7} \right\rfloor \leq d \leq \left\lceil \frac{4(n-k)+11}{7} \right\rceil . $$
\end{theorem}

To minimize $\Norm(\cdot)$ among trees of given order $n$ with $k$ peripheral vertices, we first define the special tree that will be shown to be extremal.

\begin{defi}
Given $k$ and (large) $n$, the tree $\widetilde{S}_{n,k}$ of order $n$ with $k$ peripheral vertices is obtained through the following process (Figure~\ref{fig:potato}):
\begin{itemize}
\item first join the ends of $k$ paths of length 3, resulting in the so called ``balanced starlike tree'' $S$ on $3k+1$ vertices;
\item attach $n-3k-1$ pendant vertices to one of the $k$ vertices that has normality 2 in $S$.
\end{itemize}
\end{defi}

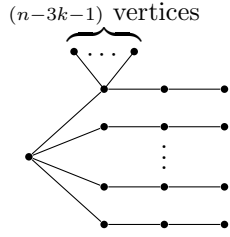
\begin{figure}[htbp]
\centering
    \begin{tikzpicture}[scale=1]
   \node[fill=black,circle,inner sep=1pt] (t0) at (-1,0) {};
        \node[fill=black,circle,inner sep=1pt] (t1) at (0,0.4) {};
        \node[fill=black,circle,inner sep=1pt] (t2) at (.8,.4) {};
        \node[fill=black,circle,inner sep=1pt] (t3) at (1.6,.4) {};

         \node[fill=black,circle,inner sep=1pt] (t5) at (0,-0.4) {};
        \node[fill=black,circle,inner sep=1pt] (t6) at (.8,-.4) {};
        \node[fill=black,circle,inner sep=1pt] (t7) at (1.6,-.4) {};

     \node[fill=black,circle,inner sep=1pt] (t9) at (0,0.9) {};
        \node[fill=black,circle,inner sep=1pt] (t10) at (.8,.9) {};
        \node[fill=black,circle,inner sep=1pt] (t11) at (1.6,.9) {};

       \node[fill=black,circle,inner sep=1pt] (t13) at (0,-0.9) {};
        \node[fill=black,circle,inner sep=1pt] (t14) at (.8,-.9) {};
        \node[fill=black,circle,inner sep=1pt] (t15) at (1.6,-.9) {};

              \node[fill=black,circle,inner sep=1pt] () at (-.4,1.4) {};
        \node[fill=black,circle,inner sep=1pt] () at (0.4,1.4) {};

      \draw (t0)--(t1);
     \draw (t0)--(t5);
     \draw (t0)--(t9);
      \draw (t0)--(t13);

     \draw (t1)--(t2);
     \draw (t2)--(t3);

  \draw (t5)--(t6);
     \draw (t6)--(t7);

      \draw (t9)--(t10);
     \draw (t10)--(t11);

      \draw (t13)--(t14);
     \draw (t14)--(t15);

   \draw (-.4,1.4)--(0,0.9);
    \draw (0,0.9)--(0.4,1.4);

      \node at (.8,0.1) {$\vdots$};
     \node at (0,1.35) {$\cdots$};
\node at (0,1.8) {$\overbrace{\hspace{2.8 em}}^{(n-3k-1) \hbox{ vertices}}$};

        \end{tikzpicture}
\caption{The tree $\widetilde{S}_{n,k}$ of order $n$ with $k$ peripheral vertices.}\label{fig:potato}
\end{figure}

\begin{theorem}
Among trees with $k$ peripheral vertices and of order $n\geq 3k+1$, $\Norm(\cdot)$ is minimized by $\widetilde{S}_{n,k}$.
\end{theorem}

\begin{proof}
We proceed by considering the largest possible numbers of vertices that can have small normalities. First there are exactly $k$ vertices with normality zero, namely the peripheral vertices.

Note that each peripheral vertex is a leaf, their unique neighbors yield at most $k$ vertices with normality 1.

We also claim that there are at most $k$ vertices with normality 2. This is because: for any vertex $v$ with normality 2, it cannot be a leaf as that will make $v$ another peripheral vertex. Then if $d(v,w) = d(u,w) =2$ for two non-peripheral vertices $v,u$ and peripheral vertex $w$, the three vertices share a common neighbor $x$. Then $u$ and $v$ are not connected through any other path in the extremal tree, resulting in vertices of eccentricity larger than $w$, a contradiction. Therefore, corresponding to each peripheral vertex there is at most one vertex of normality 2.

Lastly, the remaining $n-3k$ vertices each must have normality at least 3 and this lower bound is indeed achieved by $\widetilde{S}_{n,k}$.
\end{proof}

\begin{remark}
In fact, to minimize $\Norm(\cdot)$,  in the second step of the above process we may attach the remaining vertices (as leaves) to any number of the  $k$ vertices that has normality 2 in $S$. Hence the extremal tree here is not unique.
\end{remark}

\section{The properties of $\lambda(\cdot)$ and $\Lambda(\cdot)$}
\label{sec:lambda}

As the normality is introduced as a variation of the eccentricity, it makes sense to consider their difference $\lambda(\cdot)$ at a vertex and the sum $\Lambda(\cdot)$.

\subsection{On the behavior of $\lambda(\cdot)$}

First we show that $\lambda(v)$, as a local function on vertices of a tree $T$, behaves very much like the eccentricity. As it is maximized at all peripheral vertices (those who maximize the eccentricity) and minimized at the center (where the eccentricity is minimized).

\begin{theorem}
In a tree $T$ the maximum $\lambda(\cdot)$ is obtained at the peripheral vertices and the minimum $\lambda(\cdot)$ is obtained at the center vertices.
\end{theorem}
\begin{proof}
Since
$$\lambda(v) = \ecc(v) - \norm(v)$$
and $\ecc(\cdot)$ is maximized at the peripheral vertices (with $\ecc(v) = \diam(T)$) and $\norm(\cdot)$ is minimized at the peripheral vertices (with $\norm(v)=0$), we have
$$ \lambda(v) \leq \diam(T) $$
with equality if and only if $v \in P(T)$.

Now to minimize $\lambda(\cdot)$, we consider two cases depending on the diameter of $T$.
\begin{itemize}
\item If $d=\diam(T)$ is even, let $P$ be a path of length $d$ between peripheral vertices $u$ and $w$. Then the unique vertex $v \in C(T)$ is in the middle of the path $P$ with $\ecc(v)=\frac{d}{2}$.

Suppose now $\norm(v) = d(v, x)$ for some peripheral vertex $x$. By Proposition~\ref{prop:fact} the eccentricity at $x$ must be obtained by $d(x,u)$ or $d(x,w)$, and equals $d$ since $x \in P(T)$. Then it is easy to see that $d(v,x)=\frac{d}{2}$ and hence
$$ \lambda(v) = \ecc(v) - \norm(v) = 0 . $$
\item If $d=\diam(T)$ is odd, then by similar arguments we have the minimum $\lambda(\cdot)$ obtained at the center vertices with
$$ \lambda(v) = \ecc(v) - \norm(v) = \frac{d+1}{2} - \frac{d-1}{2} = 1 . $$
\end{itemize}
\end{proof}

\begin{remark}
From the proof it is easy to see that the peripheral vertices are the only ones with the maximum value of $\lambda(\cdot)$. On the other hand, vertices in the center of $T$ has the minimum value of $\lambda(\cdot)$ but they are not necessarily the only ones. Take, for instance, the tree $\widehat{T}_{n,d}$ from Definition~\ref{defi:tnd}, $\lambda(\cdot)$ is minimized by all vertices in the ``pendant comet''.
\end{remark}

\subsection{On the extremal values of $\Lambda(\cdot)$}

For trees of small order this can be simply treated on a case by case basis. For large $n$ we first introduce another special tree on $n$ vertices.

\begin{defi}
For large enough $n$ the tree $\widehat{S}_n$ is obtained from attaching $n-5$ pendant edges to the middle point of a path on 5 vertices. See Figure~\ref{fig hats}.
\end{defi}

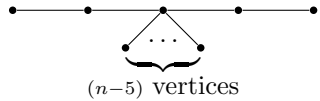
\begin{figure}[htbp]
\centering
    \begin{tikzpicture}[scale=1]
        \node[fill=black,circle,inner sep=1pt] (t1) at (0,0) {};
        \node[fill=black,circle,inner sep=1pt] (t2) at (1,0) {};
        \node[fill=black,circle,inner sep=1pt] (t3) at (2,0) {};
        \node[fill=black,circle,inner sep=1pt] (t4) at (3,0) {};

        \node[fill=black,circle,inner sep=1pt] (t5) at (4,0) {};

 \node[fill=black,circle,inner sep=1pt] (t6) at (1.5,-.5) {};
\node[fill=black,circle,inner sep=1pt] (t7) at (2.5,-.5) {};
        \draw (t1)--(t2);
        \draw (t3)--(t4);
        \draw (t2)--(t3);
        \draw  (t4)--(t5);
       \draw  (t3)--(t6);
       \draw  (t3)--(t7);
         \node at (2,-.4) {$\ldots$};

         \node at (2,-.9) {$\underbrace{\hspace{2.8 em}}_{(n-5) \hbox{ vertices}}$};
        \end{tikzpicture}
\caption{The tree $\widehat{S}_n$.}\label{fig hats}
\end{figure}

\begin{theorem}
For any tree $T$ on $n\geq 8$ vertices, we have
$$ \Lambda(T) \geq 12 $$
with equality if and only if $T\cong \widehat{S}_n$.
\end{theorem}
\begin{proof}
First it is easy to see that $\Lambda(\widehat{S}_n) = 12$.

Now for any tree $T$ of order $n\geq 8$, we consider its diameter $d$.

\begin{itemize}
\item If $d\geq 4$, consider a longest path $v_0v_1\ldots v_d$, then
$$ \lambda(v_0)=\lambda(v_d)=d $$
and
$$ \lambda(v_1)=\lambda(v_{d-1})=d-2. $$
Consequently
$$ \Lambda(T) \geq 2d + 2(d-2) \geq 12 $$
with equality if and only if $d=4$ and all other vertices has $\lambda(v)=0$. This is only the case when $T\cong \widehat{S}_n$.
\item If $d=3$, then $T$ is the so-called dumbbell and all but two of its vertices (which has $\lambda(v)=1$) are peripheral vertices with $\lambda(\cdot)$ value 3. Hence
$$ \Lambda(T) = 3(n-2) + 2 > 12 $$
for $n\geq 6$.
\item If $d=2$, then $T$ is a star and all but one vertex are peripheral vertices, with $\lambda(v) = 2$. Then
$$ \Lambda(T) = 2(n-1) > 12 $$
for $n \geq 8$.
\end{itemize}
\end{proof}

The argument in the first part of the above proof can be used to answer the same question in trees with a given order and diameter.

\begin{theorem}
Among trees on $n\geq 8$ vertices with diameter $d\geq 4$, $\Lambda(T)$ is minimized by appending pendant edges to the middle (or as close to middle as possible, depending on the parity of $d$) of a path of length $d$.
\end{theorem}

Next we consider maximizing $\Lambda(T)$ in trees. Again, it seems to be easier to first consider trees with a given diameter.

\begin{theorem}
Among trees of order $n$ with diameter $d$, $\Lambda(T)$ is maximized by a dumbbell $D(n,a,b)$ for some $a,b\geq 1$ and $a+b=n-d+1$.
\end{theorem}

\begin{proof}
Let $T$ be such an extremal tree that maximizes $\Lambda(T)$ among trees of order $n$ with diameter $d$. Now consider a longest path $P:=v_0v_1\ldots v_d$ with $v_0, v_d \in P(T)$. First we claim the following.

\begin{claim}
For any internal vertex $v_i \in V(P)$ and any peripheral vertex $w$ different from $v_0$ and $v_d$,
$$ d(v_i, w) \geq \min \{ i, d-i \}. $$
\end{claim}

\begin{proof}
Suppose, without loss of generality, that $i\leq d-i$. Note that by Proposition~\ref{prop:fact} the eccentricity at $w$ is obtained at one of $v_0$ and $v_d$, say $v_d$. Thus $d(w, v_d) = d=d(v_0,v_d)$ as $w\in P(T)$. Then it is easy to see that
$$ d(v_i,w) \geq d(v_i, v_0) = i. $$
\end{proof}

Now for $i=0,1,2, \ldots, d$, we have
$$ \lambda(v_i) = \ecc(v_i) - \norm(v_i) \leq \max \{ i, d-i \} - \min \{ i, d-i \} = |d-2i|. $$

For any other vertex $v$ we have
$$ \lambda(v) = \ecc(v) - \norm(v) \leq d - 0 = d. $$

Hence
$$ \Lambda(T) = \sum_{v \notin V(P)} \lambda(v) +  \sum_{i=0}^{d} |d-2i|    \leq (n-d-1) \cdot d + \sum_{i=0}^{d} |d-2i| $$
with equality if and only if all vertices not on $P$ are peripheral vertices, or equivalently, that $T$ is the dumbbell as described.
\end{proof}

It is easy to compute the maximum $\Lambda(T)$ depending on the parity of $d$:
\begin{itemize}
\item If $d$ is even, $\Lambda(T) \leq \frac{(2n-d)d}{2}$;
\item If $d$ is odd, $\Lambda(T) \leq \frac{(2n-d)d+1}{2}$.
\end{itemize}

Letting $f(x) = (2n-x)x$ for any given $n$ it is easy to see that $f(x)$ is increasing for $d\leq n$. Following from simple algebra we have the following.

\begin{theorem}
For any tree $T$ on $n\geq8$ vertices, we have
$$ \Lambda(T) \leq \left\lfloor \frac{n^2+1}{2} \right\rfloor $$
with equality if and only if $T$ is a path (for both odd and even $n$) or an $(n-2)$-comet (for even $n$).
\end{theorem}

\section{Concluding remarks}

In this paper we introduced the ``normality'' at a vertex, corresponding to the well studied eccentricity. Basic properties and related extremal problems are discussed. Some questions that are proposed in the context remain open.

Furthermore, it will be interesting to further investigate this new concept in different collections of trees under specific constraints (in addition to the number of peripheral vertices and diameter), and compare the extremal structures with other well-known extremal trees.

\end {document}